\documentclass[12pt]{amsart}
\usepackage[utf8]{inputenc}

\usepackage[english]{babel}
\usepackage{tabularx}
\usepackage{multirow}
\usepackage{amsmath, amssymb, amsthm, amscd,color,comment}
\usepackage{cancel}
\usepackage{cite}
\usepackage{alltt}
\usepackage[dvipsnames]{xcolor}
\usepackage{array}
\usepackage{caption} 
\usepackage{mathtools}
\usepackage{comment}

\usepackage{enumerate}
\usepackage{cancel}
\newtheorem{theorem}{Theorem}[section]
\newtheorem{definition}[theorem]{Definition}

\newtheorem{question}[theorem]{Question}
\newtheorem{remark}[theorem]{Remark}
\newtheorem{lemma}[theorem]{Lemma}
\newtheorem{corollary}[theorem]{Corollary}
\newtheorem{proposition}[theorem]{Proposition}

\newcommand{\diag}{\operatorname{diag}}

\newcommand{\rank}{\operatorname{rank}}

\newcommand{\wt}{\mathrm{W_H}}

\def\P{\mathbf P}

\def\fqs{\mathbb F_{q^2}}

\def\fq{\mathbb F_q}
\def\Fq{\mathbb F_q}

\def\rank{{\rm rank}}

\def\X{\mathbb X}
\def\Xr{\mathbb {X}_r}

\def\Xleqt{\mathbb {X}_{\leq t}}

\def \C {C_{symm}(t, m)}
\def \CA{\widehat{C}_{symm}(t, m)}
\def \Ceven {C_{symm}(2t, m)}
\def \Codd {C_{symm}(2t+1, m)}

\begin{document}
	
	\title[Symmetric Determinantal Codes ]{Linear Codes Associated to Symmetric Determinantal Varieties: General Case}
	\author{Peter Beelen }
\address{Department of Applied Mathematics and Computer Science,\newline \indent
	Technical University of Denmark,\newline \indent
	Matematiktorvet 303B, 2800 Kgs. Lyngby, Denmark.}
\email{pabe@dtu.dk }

	\author{Trygve Johnsen}
\address{Department Mathematics and Statistics,\newline \indent
	UiT: The Arctic University of Norway,\newline \indent
	Hansine Hansens veg 18, 9019 Tromsø, Norway.}
\email{trygve.johnsen@uit.no}

	\author{Prasant Singh}
\address{Department of  Mathematics,\newline \indent
	IIT Jammu,\newline \indent
	 Jammu \& Kashmir, 181221.}
\email{psinghprasant@gmail.com}	

	\date{\today}
	
\begin{abstract}
The study of linear codes over a finite field $\Fq$, for odd $q$, derived from determinantal varieties obtained from symmetric matrices of bounded rank, was initiated in \cite{BJS}. There, one found the minimum distance of the code obtained from evaluating homogeneous linear functions at all symmetric matrices with rank, which is, at most, a given even number. Furthermore, a conjecture for the minimum distance of codes from symmetric matrices with ranks bounded by an odd number was given.

In this article, we continue the study of codes from symmetric matrices of bounded rank.   A connection between the weights of the codewords of this code and $Q$-numbers of the association scheme of symmetric matrices is established. Consequently, we get a concrete formula for the weight distribution of these codes. Finally, we determine the minimum distance of the code obtained from evaluating homogeneous linear functions at all symmetric matrices with rank at most a given number,
both when this number is odd and when it is even.

	

\end{abstract}

\maketitle

\section{Introduction}
Constructing linear codes with a good minimum distance is a classical research problem in coding theory. There are many different ways to construct such codes, but one of the most interesting ways is using the language of a projective system. A projective system is a (multi)set of a projective space over finite fields. It has been proved \cite[Theorem 1.4]{TVN} that every linear code can be constructed in this way. Further, the parameters of this code can be studied from several geometric properties of the corresponding projective system. Projective varieties over finite fields with many $\Fq$-rational points are rich sources of projective systems, and several researchers have constructed different classes of codes.
Determinantal varieties in the space of different types of matrices are known to have many $\Fq$-rational points. Thus, it is natural to construct codes from such varieties and study their parameters. The study of codes associated to classical determinantal varieties was initiated by Beelen-Ghorpade-Hasan \cite{BGH} and the minimum distance of such codes was computed in a special case. Later, Beelen-Ghorpade \cite{BG} computed the minimum distance of these codes in the general case.  The problem of determining the minimum distance of these codes was studied by Ravagnani \cite{Ravagnani} in a different context. The study of codes associated with determinantal varieties in the space of skew-symmetric matrices was initiated in \cite{BS}, and the minimum distance of the corresponding codes was determined when the field under consideration is of odd characteristic.

{\it Symmetric determinantal varieties} are determinantal varieties in the space of symmetric matrices. 
For a detailed exposition of these varieties and several of their interesting properties, we refer to \cite{HTu, JLP}. Furthermore, the numbers of $\Fq$ -rational points of these varieties were determined by MacWilliams \cite{MacWilliams}. A classification of symmetric matrices is also well understood, see \cite{Carlitz}. The set of $\Fq$-rational points of symmetric determinantal varieties has been explored from different points of view. For example, the eigenvalues of association schemes of symmetric matrices are known \cite{Schmidt2020}. Several other properties of these association schemes, for example, their $P$-numbers and $Q$-numbers, are also known.

Since symmetric determinantal varieties are smooth projective varieties with many $\Fq$-rational points, it is natural to study linear codes associated with these varieties. Linear code associated with a projective system of symmetric determinantal varieties is known as {\it symmetric determinantal codes}. We initiated the study of these codes in \cite{BJS} and determined several interesting properties of them. In particular, we determined the minimum distance of symmetric determinantal codes arising from symmetric determinantal variety consisting of symmetric matrices of bounded rank bounded by an even number. Furthermore, we gave a table consisting of all possible weights of symmetric determinantal code in some special cases. Based on it, we proposed a conjecture \cite{BJS, BJS1} for the minimum distance of symmetric determinantal code, which also applies when the rank bounds are odd. In this article, we give the weight distribution of the symmetric determinantal codes, valid for both parties of the rank bound. Furthermore we prove the conjecture proposed in \cite{BJS}. In order to do so, we present a proof which is valid for all symmetric determinantal code over any finite field of odd characteristic. As a consequence, we give a new proof of the main results given in \cite{BJS}.


The results proved in this article can be interpreted in a more geometric way: The minimum distance of the code obtained from a projective system is equivalent to determining the maximum number of points of the projective system lying on a hyperplane. Thus, the minimum distance of the symmetric determinantal code also determines the maximum number of $\Fq$-rational points of symmetric determinantal variety lying on an $\Fq$-hyperplane. Moreover, we have determined the weight distribution of the symmetric determinantal code. This is equivalent to finding the cardinalities of all possible $\Fq$-hyperplane sections of the set of $\Fq$-rational points of a symmetric determinantal variety. The geometric interpretation of the results proved in this article is the following: given $\delta\in\Fq$, then for any $1\leq k\leq m$, the $\delta$-partial $k$-trace of a symmetric matrix $A=(a_{ij})$ is defined as $a_{11}+\dots+a_{k-1k-1} +\delta a_{kk}$. The weight distribution of the symmetric determinantal code is equivalent to finding the number of symmetric matrices of bounded rank and of $\delta$-partial $k$-trace nonzero for all such $\delta$. Dividing by $q-1$, one can determine the number of symmetric matrices of the bounded rank of given $\delta$-partial $k$-trace equal to $\alpha\in\Fq^*$. In fact we have computed the {\it restricted weights} of symmetric determinantal code as well and this actually determines the number of symmetric matrices of fixed rank and fixed $\delta$-partial $k$-trace.

\section{ Preliminaries}
 This section is divided into three parts. In the first part, we will discuss and give the definition of {\it symmetric determinantal code}. We will refer to \cite{BJS} for notation, and introduction to the problem of how one can determine the minimum distance of the code. In the second part, we will recall some results from Delsarte-Goethals \cite{DG1}. We will mainly be interested in the definition of generalized Krawtchouk polynomials, which arise naturally in studying association schemes of bilinear and alternating bilinear forms. In the last part of this section, we will recall some results from K.U. Schmidt \cite{Schmidt2020}. We will mainly recall the formula for $Q$-numbers of association schemes of bilinear forms in the form of a generalized Krawtchouk polynomial. With these results, we will conclude this section, and later, we will use these results to calculate the minimum distance of the symmetric determinantal code.

Throughout this article, $m$ is a fixed positive integer, $q$ is some power of an odd prime, and $\Fq$ is a finite field with $q$ elements. Let ${\bf X}=(x_{ij})$ be an $m\times m$ symmetric matrix in indeterminates $x_{ij}$. For any $0\leq t \leq m$, the {\it symmetric determinantal variety of rank $t$} is the projective variety defined by the ideal generated by all $(t+1)\times (t+1)$ minors of ${\bf X}$. The set of $\Fq$-rational points of this variety is denoted by $\mathbb{P}(\Xleqt)$. By abuse of language, we will call this set the symmetric determinantal variety of rank $t$ as well. If $\P^{{m+1\choose 2}-1}$ denotes the projective space over $\Fq$ of dimension ${m+1\choose 2}-1$, then it is not difficult to show that the symmetric determinantal variety of rank $t$ is a non-degenerate subset of $\P^{{m+1\choose 2}-1}$. Thus, using the language of projective systems  \cite[Chapter 1]{TVN}, one can talk about the linear code associated to the variety $\mathbb{P}(\Xleqt)$. The linear code associated with $\mathbb{P}(\Xleqt)$ is called the {\it symmetric determinantal code}, and we denote it by $\CA$. The length and the dimension of this code are easy to determine. The length of this code is $|\mathbb{P}(\Xleqt)|$, and the dimension of this code is ${m+1\choose 2}$. The main aim of this article is to determine the minimum distance of this code.

 One can associate a linear code corresponding to the affine cone over $\mathbb{P}(\Xleqt)$. We will denote this code by $\C$. This is shown \cite[Section 2, equation 8]{BJS} the parameters of the code $\CA$ and the code $\C$ can be studied from one another. In particular, the minimum distance of the code $\C$ is $(q-1)$ times the minimum distance of the code $\CA$. Therefore, we will mainly focus on the code $\C$ for the rest of the article. 
 For more details on this code, we will refer to\cite{BJS}.

Let $\X$ be the set of all $m\times m$ symmetric matrices over $\fq$. It is well known that $\X$ is an $\Fq$ vector space of dimension ${m+1\choose 2}$.  For $0\leq t \leq m$, we denote by $\Xleqt$, the affine algebraic variety in $\X$ given by the vanishing of $(t+1)\times (t+1)$ minors of ${\bf X}$. This is the affine cone over the symmetric determinantal variety $\mathbb{P}(\Xleqt)$. Note that this affine cone can also be defined as

\begin{equation}
\label{eq: SymmDet}
\Xleqt= \{A\in\X: \rank(A)\leq t\}
\end{equation}
For any $0\leq r\leq m$, let $\Xr$ be the set of all symmetric matrices of size $m$ and rank $r$ and let $\mu_m(r)$ be the size of this set, i.e., 
\[
\Xr = \{A\in\X: \rank(A)=r\}\qquad \text{ and }\qquad 
\mu_m(r)= |\Xr|.
\]
A formula for $\mu_m(r)$ was give by MacWilliams \cite[Theorem 2]{MacWilliams}. 
It is clear from the definition that
\[
\Xleqt= \bigcup\limits_{r=0}^t \Xr.
\]
If $\nu_m(t)$ denotes that cardinality of the affine variety $\Xleqt$, then 
\begin{equation}
\label{eq:FixRanDec}
\nu_m(t)=\sum\limits_{r=0}^t \mu_m(r).
\end{equation}

For the sake of notation, let use write $\nu_m(t)= N$ and let 
$$
\Xleqt=\{A_1, A_2,\dots, A_N\}
$$
 in some fixed order. Let $\Fq[{\bf X}]_1$ be the vector space consisting of linear homogeneous polynomials in $x_{ij}$ for $i\leq j$. It is easy to see that $\Fq[{\bf X}]_1$  is a $\Fq$ vector space of dimension ${m+1\choose 2}$ Consider the evaluation map 
 \begin{align*}
     \mathrm{Ev}: \Fq[{\bf X}]_1 &\to  \Fq^N\\
     f({\bf X})&\mapsto c_f=(f(A_1),\dots, f(A_N)).
 \end{align*}

The image of the evaluation map  $\mathrm{Ev}$ is the code $\C$. We aim to find the minimum distance of this code. Let $f({\bf X})=\sum\limits_{i\leq j} bf_{ij}x_{ij}$ be a function in $\Fq[{\bf X}]_1 $. Since the field is of odd characteristic, one can associate a symmetric matrix $B_f =(b_{ij})$ as follows
$$
b_{ij}=\begin{cases}
    f_{ij}/2 & \text{ if }i\neq j\\
    f_{ii}  & \text{ if }i= j
\end{cases}.
$$
Note that, for any symmetric matrix $A$ and  $f({\bf X})\in \Fq[{\bf X}]_1$, we have
$f(A)=\mathrm{Tr}(B_fA)$ where $\mathrm{Tr}$ denotes the trace map. Therefore, the Hamming weight of the codeword $c_f$ is given by 
\begin{equation}\label{eq:Hammingweight1}
    \wt(c_f) =\left|\{A\in\Xleqt: \mathrm{Tr}(B_fA)\neq 0\}\right|. 
\end{equation}
 It has been proved \cite[Proposition 3.3]{BJS}, that corresponding to every codeword $c_f\in \C$, one can associate a diagonal matrix $G=\diag(1,1,\dots,1, \delta,0\dots, 0)$ of the same rank as of rank $B_f$ such that 
 $$
 \wt (c_f) = \left|\{A\in\Xleqt: \mathrm{Tr}(GA)\neq 0\}\right|.
 $$
Further, it has been shown that the weight $\wt(c_f)$ depends only on the rank of the matrix $G$ and $\delta$. 
Therefore, for any $0\leq k\leq m$ and $\delta\in\Fq$, we define 
$$
W_k^\delta(t, m):=\left|\{A\in\Xleqt : \mathrm{Tr}(GA)\neq 0\} \right|
$$
where $G=\diag(1,1,\dots,1, \delta,0\dots, 0)$ is the matrix of rank $k$. Therefore, now the problem of determining the minimum distance of the code $\CA$ is equivalent to answering the following question:
\begin{question}\label{Que:question1}
What is the minimum of  $W_k^\delta(t, m)$ when the minimum is over $k$ and $\delta$?
\end{question}

In \cite{BJS}, we answered the above question when $t$ is even. In fact, we proved that in the case $t$ is even, $W_1^\delta(t, m)$ is independent of $\delta$ and is the minimum among all $W_k^\delta(t, m)$. Moreover, we explicitly gave a formula for the precise value of   $W_1^\delta(t, m)$. The problem is open in the general case, i.e., when $t$ is odd, the problem is open. However, we conjectured \cite[Conjecture 4.3]{BJS} that in this case ($t$ is odd), $W_2^\delta(t, m)$  is going to be the minimum where $-\delta$ is a square in the field. We will answer to \ref{Que:question1} in the general case.

Now, we will recall the definition of generalized Krawtchouk polynomial. This part of the article is taken from \cite{DG1} although, to be consistent, we have changed some notations. Let $m$ and $q$ be as earlier. Set $n=\lfloor m/2\rfloor$. Let $\Omega$ be the set of all $m\times m$ skew-symmetric matrices over the field $\Fq.$ For $0\leq s\leq n$, define
$$
\Omega_s =\{A\in \Omega: \rank(A)= 2s\}.
$$
Let $\chi$ be a non-trivial character of the additive group $(\Fq, +)$. For two skew-symmetric matrices $A=(a_{ij})$ and $B=(b_{ij})$, we define
$$
[A, B] = \chi\left(\sum_{i<j}a_{ij}b_{ij}\right).
$$
Also, for a fixed $A\in \Omega_s$ and $0\leq r\leq n$, we define
\begin{equation}
    \label{eq:F function}
    F^{(m)}_r(s)=\sum_{B\in\Omega_r}[A, B], 
\end{equation}
It has been proved that the value of $ F^{(m)}_r(s)$ is independent of the matrix $A$ and the character $\chi$. It has been proved \cite[ Equation 15]{DG1} that 

\begin{equation}\label{eq:Fdef}
F_r^{(m)}(s)=\sum_{j=0}^r (-1)^{r-j}q^{(r-j)(r-j-1)}\left[ \begin{array}{c} n-j \\ n-r\end{array}\right]\left[ \begin{array}{c} n-s \\ j\end{array}\right]c^j.
\end{equation}
 where $c:=q^{m(m-1)/(2n)}$ and (be aware of the notation)
$$
\left[ \begin{array}{c} n \\ k\end{array}\right]=\prod_{i=0}^{k-1}\frac{q^{2n}-q^{2i}}{q^{2k}-q^{2i}}.
$$

In the last section, we will recall some results from \cite{Schmidt2020} related to $Q$-numbers of the association schemes of symmetric matrices. We begin by recalling that the classification of symmetric matrices (the quadratic forms) is well-known over finite fields. For example, a symmetric matrix of odd rank is always of parabolic type, but a symmetric matrix of even rank can be either hyperbolic or elliptic type. For a more detailed analysis of the classification of quadratic forms, we refer to \cite{Carlitz, Hir}. Therefore, we will say that a symmetric matrix is of type $1$ if it is hyperbolic and type $-1$ if it is elliptic. For convention, parabolic matrices, i.e., symmetric matrices of odd rank, will be called of type $0$. Therefore, for $0 \le r \le m$ and $\tau \in \{-1,0,1\}$, we define
$$\X_{r,\tau}=\{B \in \X \mid \mathrm{rank}(B)=r, \mathrm{type}(B)=\tau\}$$
Furthermore we define for $A,B \in \X$:
$$\langle A,B \rangle :=\chi(\mathrm{Tr}(A\cdot B)),$$
where $\mathrm{Tr}$ denotes the trace and $\chi: \fq \to \mathbb{C}$ is a character of $(\fq,+)$. It is well known (See \cite[Theorem 5.8]{LN}) that all possible characters of $(\Fq, +)$ are of the form
$\chi_\alpha(\beta)=\zeta_p^{\mathrm{Tr}_{\fqs/\mathbb{F}_p}(\alpha\cdot \beta)}$, where $\zeta_p$ is a $p$-root of unity in $\mathbb{C}$ and $\alpha \in \fq$.
The character $\chi_0$ is called the trivial character. It is well known that for any $x \in \fq$:
\begin{equation}\label{eq:charsum}
\sum_{\chi_\alpha} \chi_\alpha(x)=\left\{
\begin{array}{rl}
q & \text{ if $x=0$,}\\
0 & \text{otherwise.}
\end{array}
\right.
\end{equation}

\begin{definition}
Let $\chi: \fq \to \mathbb{C}$ be a character of $(\fq,+)$. Then we define
$$
Q^{\chi}_{r,\epsilon}(s,\tau)=\sum_{B \in \X_{r,\epsilon}}\langle A, B \rangle\quad \text{for $A \in \X_{s,\tau}$.}
$$
\end{definition}

It is not difficult to show that the value of $Q^{\chi}_{r,\epsilon}(s,\tau)$ does not depend on the choice of $A \in X_{s,\epsilon}$ and depends only on $s$ and $\epsilon$. Since $\mathrm{Tr}(A\cdot B)=\mathrm{Tr}(B\cdot A)$, one may replace $\langle A, B \rangle$ with $\langle B, A \rangle$ in the definition of  $Q^{\chi}_{r,\epsilon}(s,\tau)$. If $\chi$ is a trivial character,  then it is not difficult to show that $Q^{\chi}_{r,\epsilon}(s,\tau)=|X_{r,\tau}|=\mu_{r,\tau}{(m)}$. If $r$ is odd, the value of $\epsilon$ is irrelevant (there is only one type in that case). Similarly, if $s$ is odd, the value of $\tau$ is irrelevant. We may simplify the notation in these cases and omit $\epsilon$ and/or $\tau$. Note that the value of $\mu_{r,\tau}{(m)}$ is given by \cite[Proposition 2.4]{Schmidt2020}
$$\mu_{2r+1}(m)=\frac{1}{q^r}\frac{\prod_{i=0}^{2r}(q^m-q^i)}{\prod_{i=0}^{r-1}(q^{2r}-q^{2i})},$$
$$\mu_{2r,\tau}(m)=\frac{q^r+\tau}{2}\frac{\prod_{i=0}^{2r-1}(q^m-q^i)}{\prod_{i=0}^{r-1}(q^{2r}-q^{2i})}$$
and
$$\mu_{2r}(m)=\mu_{2r,1}(m)+\mu_{2r,-1}(m).$$

The $Q$-numbers defined above are known \cite[Theorem 3.1]{Schmidt2020}. We will conclude this section by recalling this result. We assume that $\chi$ is a nontrivial character and $q$ is odd. 

\begin{theorem}\label{thm:Qnumber} The $Q$-numbers of the association schemes corresponding to symmetric matrices are given by
$$Q_{0,1}^{\chi}(i)=1 \quad \text{and} \quad Q_{k}^{\chi}(0)=\mu_k(m).$$

\begin{equation}\label{eq:Qoddodd}
Q_{2r+1}^{\chi}(2s+1)=-q^{2r}F_r^{(m-1)}(s)
\end{equation}
\begin{equation}\label{eq:Qoddeven}
Q_{2r+1}^{\chi}(2s,\tau)=-q^{2r}F_r^{(m-1)}(s-1)+\tau q^{m-s+2r}F_r(s-1)^{(m-2)}
\end{equation}
\begin{equation}\label{eq:Qevenodd}
Q_{2r,\epsilon}^{\chi}(2s+1)=\frac12 q^{2r}F_r^{(m-1)}(s)+\epsilon \frac{q^r}{2}F_r^{(m)}(s)
\end{equation}
\begin{equation}\label{eq:Qeveneven}
Q_{2r,\epsilon}^{\chi}(2s,\tau)=\frac12 \left[ q^{2r}F_r^{(m-1)}(s-1)-\tau q^{m-s+2r-2} F_{r-1}^{(m-2)}(s-1)\right]+\epsilon \frac{q^r}{2}F_r^{(m)}(s).
\end{equation}
In particular, the value of 
$Q_{k,\epsilon}^{\chi}(i,\tau)$ is the same for all nontrivial  characters $\chi$ of $(\fq,+)$. 
\end{theorem}
Therefore, for any nontrivial character $\chi$ of $(\Fq, +)$, we will omit the $\chi$ while writing the $Q$ numbers and just write $Q_{r,\epsilon}(s,\tau)$ instead of $Q^{\chi}_{r,\epsilon}(s,\tau)$.

\section{The Minimum Distance of The Code $\C$}
This is the article's final section; in this, we will determine the minimum distance of the code $\C.$ To do so, we will use the $Q$-numbers of association schemes of symmetric matrices and the expressions for them given in Theorem \ref{thm:Qnumber}. We will determine the minimum distance of the code $\C$ for a general $t$. This will give an alternative proof of the minimum distance computation given in \cite{BJS}. Recall that from \eqref{eq:Hammingweight1}, corresponding to every codeword $c_f\in\C$ one can associate a symmetric matrix $B_f$ such that 
$$
\wt(c_f) =\left|\{A\in\Xleqt: \mathrm{Tr}(B_f.A)\neq 0\}\right|.
$$
We also mentioned that it has been proved  \cite[Proposition 3.3]{BJS}, that corresponding to every codeword $c_f\in \C$, one can associate a diagonal matrix $G=\diag(1,1,\dots,1, \delta,0\dots, 0)$ of the same rank as of rank $B_f$ such that 
 $$
 \wt (c_f) = \left|\{A\in\Xleqt: \mathrm{Tr}(GA)\neq 0\}\right|.
 $$

Moreover, we will also show that the Hamming weight of codeword $c_f$ depends only on the type of matrix $B_f$. In other words, if the rank of $B_f$ is odd, then the weight $\wt(c_f)$ does also not depend on $\delta $ defined above. But before proceeding any further, note that there is a one-to-one correspondence between the codewords $c_f\in \C$ and symmetric matrices $B=(b_{ij})\in\X$. For example, corresponding to $B=(b_{ij})\in\X$, one can associate $f({\bf x})=\sum\limits_{i\leq j} b_{ij}x_{ij}\in \Fq[{\bf X}]_1$ and $c_B\in\C$ be the corresponding codeword.  Let us define the Hamming $\wt(B)$ as the codeword $c_B.$ For any $B\in\X$ and $0\leq r\leq m$, we define the { \it restricted weight}
\begin{equation}
    \label{eq:restweight}
    \mathrm{w}_B(r,m)= \left|\{A\in\X_r: \mathrm{Tr}(BA)\neq 0\}\right|.
\end{equation}
 Clearly, the Hamming weight of the codeword $c_B$ is the sum of $\mathrm{w}_B(r,m)$, i.e.,
\begin{equation}
   \label{eq:Wwrel}  
    \wt(c_B)= \sum\limits_{r=1}^t \mathrm{w}_B(r,m).
\end{equation}

In the next lemma, we establish a connection between the $Q$-numbers and the restricted weights
\begin{lemma}\label{lem:charsumcount}
For a matrix $B\in\X_{k,\tau}$,
$$\left|\{A \in \X_{r,\epsilon} \mid \mathrm{tr}(A \cdot B)=0 \}\right|=\frac{1}{q}\sum_{\chi} Q_{r,\epsilon}(k,\tau)$$
\end{lemma}

\begin{proof}
Let $\chi$ be a nontrivial character of $(\Fq,+)$, then we have
\begin{eqnarray*}
\sum_{\chi} Q_{r,\epsilon}(k,\tau) & = & \sum_{\chi} \sum_{A \in X_{r,\epsilon}}\langle A, B \rangle\\
 & = & \sum_{A \in X_{r,\epsilon}}\sum_{\chi} \chi(\mathrm{Tr}(AB)) \\
 & = & \sum_{\substack{A \in X_{r,\epsilon}\\\mathrm{Tr}(BA)=0 }} q \qquad\qquad\qquad\quad\quad(\text{  by equation \eqref{eq:charsum} }).
\end{eqnarray*}
The result follows now.
\end{proof}
 The next corollary is a simple consequence of the above Lemma. 
 \begin{corollary}
     \label{cor:charsumcount}
      Let $B\in\X_{k,\tau}$, and let $c_B\in\C$ be the corresponding codeword. Then 
      $$
       \mathrm{w}_B(r,m) =\left| \X_{r}\right| - \frac{1}{q}\sum_{\epsilon}\sum_{\chi} Q_{r,\epsilon}(k,\tau).
      $$
 \end{corollary}

\begin{proof}
Let  $B\in\X_{k,\tau}$, and let $c_B\in\C$ be the corresponding codeword. Then 
  \begin{eqnarray*} 
\mathrm{w}_B(r,m) & =& \left|\{A\in\X_r: \mathrm{Tr}(BA)\neq 0\}\right|\\
& = & \sum_{\epsilon} \left|\{A\in\X_{r, \epsilon}: \mathrm{Tr}(BA)\neq 0\}\right|\\
& = & \left| \X_{r}\right| - \sum_{\epsilon} \left|\{A\in\X_{r, \epsilon}: \mathrm{Tr}(BA) = 0\}\right|\\ 
& =& \left| \X_{r}\right| - \frac{1}{q}\sum_{\epsilon}\sum_{\chi} Q_{r,\epsilon}(k,\tau),
  \end{eqnarray*}
  where the last equality follows from the Lemma \ref{lem:charsumcount}.
\end{proof}
The next corollary implies that the restricted weights $\mathrm{w}_B(r,m)$ depend only on the type of the matrix $B$.

\begin{corollary}
     \label{cor:charsumcount2}
      Let $B\in\X_{k,\tau}$, and let $c_B\in\C$ be the corresponding codeword. Then, the restricted weight $\mathrm{w}_B(r,m)$  of the codeword $c_B$ depends only on $k$ and $\tau. $  Consequently, the Hamming weight $\wt(c_B)$ depends only on the rank $k$ and type $\tau$ of $B$.

\end{corollary}
\begin{proof}
This follows from Corollary \ref{cor:charsumcount} as $Q$-numbers $Q_{r,\epsilon}(k,\tau)$ depends on $k$ and $\tau$ and not on the matrix $B$ itself.
    \end{proof}

\begin{remark}
\label{rmk:1}
  Motivated by the above corollary, we define 
$$
\mathrm{w}_k^{\tau}(r, m) =\mathrm{w}_B(r,m) , \quad W_k^\tau(t, m) =\wt(c_B)
$$
for some matrix $B\in\X_{k, \tau}$.  Clearly, from Corollary \ref{cor:charsumcount2} we have that the number $\mathrm{w}_k^{\tau}(r, m)$ and $ W_k^\tau(t, m)$ are well defined numbers. Also, $ W_k^\tau(t, m)$ are all possible weights of codewords of $\C.$ Also, when $k$ is odd, there is only one type of symmetric matrices therefore if $k=2\ell+1$, we will write $\mathrm{w}_{2\ell+1}(r, m)$ instead of $\mathrm{w}_{2\ell+1}^{\tau}(r, m)$. Similarly, we will write $\mathrm{W}_{2\ell+1}(r, m)$ instead of $\mathrm{W}_{2\ell+1}^{\tau}(r, m)$
\end{remark}
 
 The next proposition is an important step in getting a formula for weights of the code $\C$. In this, we will give a formula for the restricted weights of the code $\C$.
 
\begin{proposition}
    \label{prop:restweight}
    For any $r$ and $\ell$, the restricted weights of the code $\C$ are given by
    $$
    \mathrm{w}_{2\ell+1}(2r+1,m)=\frac{q-1}{q}\left( \mu_{2r+1}(m)+ q^{2r}F_r^{(m-1)}(\ell)\right),
    $$
$$
\mathrm{w}_{2\ell}^\tau(2r+1,m)=\frac{q-1}{q}\left( \mu_{2r+1}(m)+q^{2r}F_r^{(m-1)}(\ell-1)-\tau q^{m-\ell+2r}F_r^{(m-2)}(\ell-1)\right),
$$
$$
\mathrm{w}_{2\ell+1}(2r,m)=\frac{q-1}{q}\left( \mu_{2r}(m)-q^{2r}F_r^{(m-1)}(\ell)\right),
$$
and
$$
\mathrm{w}_{2\ell}^\tau(2r,m)=\frac{q-1}{q}\left( \mu_{2r}(m)-q^{2r}F_r^{(m-1)}(\ell-1)+\tau q^{m-\ell+2r-2}F_{r-1}^{(m-2)}(\ell-1)\right).
$$
\end{proposition}

\begin{proof}
 Recall that there are exactly $q$ characters of $(\Fq, +)$.  Among them, one character is trivial, and for all nontrivial characters, we have seen in Theorem \ref{thm:Qnumber} that the $Q$ numbers are the same.  Also, from Corollary \ref{cor:charsumcount2}, and the fact that for odd rank, there is only one type of symmetric matrices, we have    
 \begin{eqnarray*} 
\mathrm{w}_{2\ell+1}(2r+1, m) & =& \left| \X_{2r+1}\right| - \frac{1}{q}\sum_{\epsilon}\sum_{\chi} Q_{2r+1,\epsilon}(2\ell+1,\tau)\\
& = &\mu_{2r+1}(m) - \frac{1}{q}\sum_{\chi} Q_{2r+1}(2\ell+1)\\
& = &\mu_{2r+1}(m) - \frac{1}{q}Q_{2r+1}^{0}(2\ell+1)-\frac{1}{q}\sum_{\chi\neq 0} Q_{2r+1}(2\ell+1)\\
& = &\mu_{2r+1}(m) - \frac{1}{q}\mu_{2r+1}(m)-\frac{q-1}{q} Q_{2r+1}(2\ell+1)\\
& = &\dfrac{q-1}{q}\left(\mu_{2r+1}(m) +\; q^{2r}F_r^{(m-1)}(\ell)\right)
  \end{eqnarray*}
  where the last equality follows from the equation \eqref{eq:Qoddodd}. Similarly,
  \begin{eqnarray*} 
\mathrm{w}_{2\ell}^\tau(2r+1, m) & =& \left| \X_{2r+1}\right| - \frac{1}{q}\sum_{\epsilon}\sum_{\chi} Q_{2r+1,\epsilon}(2\ell+1,\tau)\\
& = &\mu_{2r+1}(m) - \frac{1}{q}\sum_{\chi} Q_{2r+1}(2\ell, \tau)\\
& = &\mu_{2r+1}(m) - \frac{1}{q}Q_{2r+1}^{0}(2\ell, \tau)-\frac{1}{q}\sum_{\chi\neq 0} Q_{2r+1}(2\ell, \tau)\\
& = &\mu_{2r+1}(m) - \frac{1}{q}\mu_{2r+1}(m)-\frac{q-1}{q} Q_{2r+1}(2\ell, \tau)\\
& = &\dfrac{q-1}{q}\left(\mu_{2r+1}(m)+q^{2r}F_r^{(m-1)}(\ell-1)-\tau q^{m-\ell+2r}F_r^{(m-2)}(\ell-1)\right)
  \end{eqnarray*}
  where again the last equality follows from the equation \eqref{eq:Qoddeven}. In the same way, 
  \begin{eqnarray*} 
\mathrm{w}_{2\ell+1}(2r, m) & =& \left| \X_{2r}\right| - \frac{1}{q}\sum_{\epsilon}\sum_{\chi} Q_{2r,\epsilon}(2\ell+1)\\
& = &\mu_{2r}(m) - \frac{1}{q}\sum_{\chi} Q_{2r, 1}(2\ell+1) -\frac{1}{q}\sum_{\chi} Q_{2r,-1}(2\ell+1)\\
& = &\mu_{2r}(m) - \frac{1}{q}Q_{2r,1}^{0}(2\ell+1)-\frac{1}{q}\sum_{\chi\neq 0} Q_{2r, 1}(2\ell+1)\\
&  - &\frac{1}{q}Q_{2r, -1}^{0}(2\ell+1)-\frac{1}{q}\sum_{\chi\neq 0} Q_{2r, -1}(2\ell+1)\\
& = &\mu_{2r}(m) - \frac{1}{q}\mu_{2r, 1}(m)-\frac{q-1}{q} Q_{2r, 1}(2\ell+1)\\
&-& \frac{1}{q}\mu_{2r, -1}(m)-\frac{q-1}{q} Q_{2r, -1}(2\ell+1)\\
& = &\mu_{2r}(m) - \frac{1}{q}\mu_{2r}(m)-\frac{q-1}{q} \left(Q_{2r, 1}(2\ell+1) +Q_{2r, -1}(2\ell+1)\right)\\
& = &\dfrac{q-1}{q}\left(\mu_{2r}(m) -\; q^{2r}F_r^{(m-1)}(\ell)\right)
  \end{eqnarray*}
  where again the last equality follows from the equation \eqref{eq:Qevenodd}. Finally,
  
\begin{eqnarray*} 
\mathrm{w}_{2\ell}^\tau(2r, m) & =& \left| \X_{2r}\right| - \frac{1}{q}\sum_{\epsilon}\sum_{\chi} Q_{2r,\epsilon}(2\ell, \tau)\\
& = &\mu_{2r}(m) - \frac{1}{q}\sum_{\chi} Q_{2r, 1}(2\ell,\tau) -\frac{1}{q}\sum_{\chi} Q_{2r, 1}(2\ell, \tau)\\
& = &\mu_{2r}(m) - \frac{1}{q}Q_{2r,1}^{0}(2\ell,\tau)-\frac{1}{q}\sum_{\chi\neq 0} Q_{2r, 1}(2\ell,\tau)\\
&  - &\frac{1}{q}Q_{2r, -1}^{0}(2\ell,\tau)-\frac{1}{q}\sum_{\chi\neq 0} Q_{2r, -1}(2\ell,\tau)\\
& = &\mu_{2r}(m) - \frac{1}{q}\mu_{2r, 1}(m)-\frac{q-1}{q} Q_{2r, 1}(2\ell,\tau)\\
&-& \frac{1}{q}\mu_{2r, -1}(m)-\frac{q-1}{q} Q_{2r, -1}(2\ell,\tau)\\
& = &\mu_{2r}(m) - \frac{1}{q}\mu_{2r}(m)-\frac{q-1}{q} \left(Q_{2r, 1}(2\ell,\tau) +Q_{2r, -1}(2\ell,\tau)\right)\\
& = &\dfrac{q-1}{q}\left( \mu_{2r}(m)-q^{2r}F_r^{(m-1)}(\ell-1)+\tau q^{m-\ell+2r-2}F_{r-1}^{(m-2)}(\ell-1)\right).
  \end{eqnarray*}
  where the last equality follows from the equation \eqref{eq:Qeveneven}.
\end{proof}

The next two theorems are among the main results of this article. In these, we give an explicit formula for all possible weights of the code $\C$.

\begin{theorem}
    \label{thm:Weven}
    Let $0\leq t\leq n=\lfloor m/2\rfloor$. All possible weights of code $\Ceven$ are 
    $$
W_{2\ell+1}(2t,m)=\dfrac{q-1}{q}\sum\limits_{s=0}^{2t}\mu_s(m) -\dfrac{q-1}{q} q^{2t} F_t^{(m-1)}(\ell)
$$
and
$$
W_{2\ell}^\tau(2t, m)=\dfrac{q-1}{q}\sum\limits_{s=0}^{2t}\mu_s(m) -\dfrac{q-1}{q} q^{2t} F_t^{(m-1)}(\ell-1).
$$
\end{theorem}
\begin{proof}
Suppose $0\leq \ell\leq \lfloor(m-1)/2\rfloor$. Then we know from Remark \ref{rmk:1} that the weight $W_{2\ell+1}^\tau(2t, m)$ is independent of $\tau.$ From equation \eqref{eq:Wwrel}, we have
$$
W_{2\ell+1}(2t,m) =\sum\limits_{s=0}^{2t} w_{2\ell+1}(s, m)
$$
Using Proposition  \ref{prop:restweight}, we find
\begin{eqnarray*}
W_{2\ell+1}(2t,m) &=&\sum\limits_{s=0}^{2t} \mathrm{w}_{2\ell+1}(s, m)\\
&=& \sum\limits_{s=0}^{t}  \mathrm{w}_{2\ell+1}(2s, m) + \sum\limits_{s=0}^{t-1} \mathrm{w}_{2\ell+1}(2s+1, m)\\
&=&\dfrac{q-1}{q}\left(  \sum\limits_{s=0}^{t} \mu_{2s}(m) - q^{2s} F_s^{(m-1)}(\ell)    \right)\\
&+& \dfrac{q-1}{q}\left(  \sum\limits_{s=0}^{t-1} \mu_{2s+1}(m) + q^{2s} F_s^{(m-1)}(\ell)    \right)\\
&=& \dfrac{q-1}{q}\sum\limits_{s=0}^{2t}\mu_s(m) -\dfrac{q-1}{q} q^{2t} F_t^{(m-1)}(\ell).
\end{eqnarray*}

On the other hand, if $0\leq \ell\leq \lfloor m/2\rfloor$, then again using equation \eqref{eq:Wwrel}, we have
$$
W_{2\ell}^\tau(2t,m) =\sum\limits_{s=0}^{2t} \mathrm{w}_{2\ell}^\tau(s,m).
$$

 With the help of Proposition  \ref{prop:restweight}, we get
 \begin{eqnarray*}
W_{2\ell}^\tau(2t,m) &=&\sum\limits_{s=0}^{2t} \mathrm{w}_{2\ell}^\tau(s, m)\\
&=& \sum\limits_{s=0}^{t}  \mathrm{w}_{2\ell}^\tau(2s, m) + \sum\limits_{s=0}^{t-1} \mathrm{w}_{2\ell}^\tau(2s+1, m)\\
&=&\dfrac{q-1}{q}\left(  \sum\limits_{s=0}^{t} \mu_{2s}(m)-q^{2s}F_s^{(m-1)}(\ell-1)+\tau q^{m-\ell+2s-2}F_{s-1}^{(m-2)}(\ell-1)  \right)\\
&+& \dfrac{q-1}{q}\left(  \sum\limits_{s=0}^{t-1} \mu_{2s+1}(m) + q^{2s}F_s^{(m-1)}(\ell-1)-\tau q^{m-\ell+2s}F_{s-1}^{(m-2)}(\ell-1)  \right)\\
&=&\dfrac{q-1}{q}\sum\limits_{s=0}^{2t}\mu_s(m) -\dfrac{q-1}{q} q^{2t} F_t^{(m-1)}(\ell-1).
\end{eqnarray*}
This completes the proof of the theorem. 

\end{proof}
In \cite{BJS}, we have computed the minimum distance of the code $\Ceven$. Furthermore, in \cite{BJS1}, we have given a table of weight distribution of code $\C$ for small values of $m$. This table was computed in SAGE. In the table, one can easily make some observations. For example,  the weight $W^\tau_k(2t, m)$ does not depend on $\tau$. In the next two corollaries, we will prove other interesting results for the code $\Ceven$.

\begin{corollary}
    \label{cor:Weven1}
    For the code $\Ceven$, the weights $W_k^\tau(2t, m)$ are independent of $\tau$
\end{corollary}
\begin{proof}
    When $k$ is odd, we already have discussed that the weight $W_k^\tau(t, m)$ does not depend on $\tau$ whether $t$ is odd or even. Now the $W_{2\ell}^\tau(2t, m)$ are independent of $\tau$ and the result follows from the expression of $W_{2\ell}^\tau(2t, m)$ in Theorem \ref{thm:Weven}.
\end{proof}
    \begin{corollary}
    \label{cor:Weven2}
    For the code $\Ceven$ and $1\leq \ell \leq \lfloor m/2\rfloor$, the following holds
    $$
    W_{2\ell-1}(2t, m) =W_{2\ell}(2t, m)
    $$
\end{corollary}
\begin{proof}
    It follows immediately from the formula for $W_{2\ell}(2t, ,m)$ and $W_{2\ell+1}(2t, ,m)$ in Theorem \ref{thm:Weven}.
\end{proof}
In the next theorem we get the weight distribution of the code $\C$ for odd $t$.

\begin{theorem}
    \label{thm:Wodd}
    Let $0\leq t\leq n=\lfloor (m-1)/2\rfloor$. All possible weights of code $\Codd$ are 
$$
W_{2\ell+1}(2t+1,m)=\frac{q-1}{q}\sum_{r=0}^{2t+1}\mu_r(m).
$$
and
$$W_{2\ell}^\tau(2t+1,m)=W_{1}(2t+1,m)-\tau \frac{q-1}{q}\left(q^{m-\ell+2t}F_t^{(m-2)}(\ell-1)\right)$$.
\end{theorem}
\begin{proof}
The proof of the theorem is exactly as of Theorem \ref{thm:Weven}. Again using Proposition \ref{prop:restweight}, we have
\begin{eqnarray*}
W_{2\ell+1}(2t+1,m) &=&\sum\limits_{s=0}^{2t+1} \mathrm{w}_{2\ell+1}(s, m)\\
&=& \sum\limits_{s=0}^{t}  \mathrm{w}_{2\ell+1}(2s, m) + \sum\limits_{s=0}^{t} \mathrm{w}_{2\ell+1}(2s+1, m)\\
&=&\dfrac{q-1}{q}\left(  \sum\limits_{s=0}^{t} \mu_{2s}(m) - q^{2s} F_s^{(m-1)}(\ell)    \right)\\
&+& \dfrac{q-1}{q}\left(  \sum\limits_{s=0}^{t} \mu_{2s+1}(m) + q^{2s} F_s^{(m-1)}(\ell)    \right)\\
&=& \dfrac{q-1}{q}\sum\limits_{s=0}^{2t}\mu_s(m).
\end{eqnarray*}

Similarly, 
\begin{eqnarray*}
W_{2\ell}^\tau(2t+1,m) &=&\sum\limits_{s=0}^{2t+1} \mathrm{w}_{2\ell}^\tau(s, m)\\
&=& \sum\limits_{s=0}^{t}  \mathrm{w}_{2\ell}^\tau(2s, m) + \sum\limits_{s=0}^{t} \mathrm{w}_{2\ell}^\tau(2s+1, m)\\
&=&\dfrac{q-1}{q}\left(  \sum\limits_{s=0}^{t} \mu_{2s}(m)-q^{2s}F_s^{(m-1)}(\ell-1)+\tau q^{m-\ell+2s-2}F_{s-1}^{(m-2)}(\ell-1)  \right)\\
&+& \dfrac{q-1}{q}\left(  \sum\limits_{s=0}^{t} \mu_{2s}(m) + q^{2s}F_s^{(m-1)}(\ell-1)-\tau q^{m-\ell+2s}F_{s-1}^{(m-2)}(\ell-1)  \right)\\
&=&\dfrac{q-1}{q}\sum\limits_{s=0}^{2t}\mu_s(m) -\tau \frac{q-1}{q}\left(q^{m-\ell+2t}F_t^{(m-2)}(\ell-1)\right).
\end{eqnarray*}
This completes the proof of the theorem.
\end{proof}
 Now, we are ready to prove the main theorem of this article. In the next theorem, we determine the minimum distance of the code $\C$. Recall that we have proved \cite[Theorem 4.11]{BJS} that $W_1(2t, m)$ is the minimum distance of the code $\Ceven.$

\begin{theorem}
    \label{thm:MinDist}
    The minimum distance of the code $\Ceven$ is $W_1(2t, m)=W_2(2t,m)$  and the minimum distance of the code $\Codd$ is $W_2^{1}(2t+1, m)$.
\end{theorem}
\begin{proof}
From Theorem \ref{thm:Weven}, we have
\begin{eqnarray*}
W_k(2t, m) - W_1(2t, m) =\begin{cases}
\frac{q-1}{q}q^{2t}\left(F_t^{(m-1)}(0)-F_t^{(m-1)}(\ell)\right),  & \text{ if }k=2\ell+1\\
\frac{q-1}{q}q^{2t}\left(F_t^{(m-1)}(0)-F_t^{(m-1)}(\ell-1)\right)  & \text{ if }k=2\ell\\   
\end{cases}
\end{eqnarray*}
Therefore, $W_1(2t, m)$ is the minimum distance of the code $\Ceven$ if and only if $$
F_t^{(m-1)}(0)-F_t^{(m-1)}(\ell)\geq 0 \text{ for all $\ell\geq 0$. }
$$ 
On the other hand, for the code $\Codd$, we have
$$
W_{2\ell+1}(2t+1, m) - W_2^{1}(2t+1, m) =\frac{q-1}{q}q^{m-1+2t}F_t^{(m-2)}(0)
$$
 and 
 $$
 W_{2\ell}^\tau(2t+1, m) - W_2^{1}(2t+1, m) =\frac{q-1}{q}\left(q^{m-1+2t}F_t^{(m-2)}(0)-\tau q^{m-\ell+2t}F_t^{(m-2)}(\ell-1)\right).
 $$
 Hence,   $W_2^{1}(2t+1, m)$ is the minimum distance of the code $\Codd$ if and only if 
 \begin{enumerate}
     \item $F_t^{(m-2)}(0)\geq 0$, and 
     \item  $F_t^{(m-2)}(0)\geq \tau q^{-\ell+1}F_t^{(m-2)}(\ell-1)$ for all $\ell\geq 1$.
 \end{enumerate}
 Since $\tau\in\{1, -1\}$ and $-\ell+1\leq 0$, to complete the proof it is enough to show that 
 $$
 F_t^{(m)}(0) \geq \left| F_t^{(m)}(\ell)\right|
 $$
 for all $m$ and $\ell\geq 0$. From equation \eqref{eq:F function}, we have for any skew-symmetric matrix $A\in\Omega_\ell$
\begin{align*}
\left| F^{(m)}_t(\ell)\right|&=\left |\sum_{B\in\Omega_t}[A, B] \right| \\
       &\leq \sum_{B\in\Omega_t}\left|[A, B] \right| \\
      = & \sum_{B\in\Omega_t}\left |\chi\left(\sum_{i<j}a_{ij}b_{ij}\right)\right| \\
        &\leq \sum_{B\in\Omega_t}1 \\
       &=\left |\Omega_t\right| = F^{(m)}_t(0) \qedhere
\end{align*}
where the last inequality follows from the fact that $\chi $ is a non-trivial character.
\end{proof}

\begin{remark}
\rm{The type $\tau=1$ appearing as the superscript $1$ in $W_2^{1}(2t+1, m)$ in Theorem \ref{thm:MinDist} tells us that the 
diagonal matrix $G=diag(1,\delta,0,\cdots,0)$ described on p. 4-5 is of hyperbolic (as opposed to elliptic) type. Here we have used the notation of \cite{Schmidt2020}. That this $G$ is hyperbolic is the same as saying that $-\delta$ is a square, which typically happens when $\delta=-1.$ In \cite{BJS} we indexed weights as $W_k^{\delta}(a,m)$. Hence the weight $W_k^{1}(2t+1,m)$ appearing in Theorem \ref{thm:MinDist} is called $W_k^{-1}(2t+1,m)$ in the notation of \cite{BJS} and \cite{BJS1} (and also $W_k^{1}(2t+1,m)$ only if $-1$ is a square in $\Fq).$}

We observe that the last part of Theorem \ref{thm:MinDist} then confirms Conjecture 4.13 of \cite{BJS} and \cite{BJS1}.
Moreover we see from Theorems \ref{thm:Weven} and \ref{thm:Wodd} that the codes $C_{symm}(2t, m)$ have at most $[\frac{m+1}{2}]$ different weights, and the codes $C_{symm}(2t+1, m)$ have at most $2[\frac{m}{2}]+1$ different weights, as indicated by the final tables in \cite{BJS1}.
    
\end{remark}

\begin{corollary}
From Theorems \ref{thm:Weven} and \ref{thm:MinDist} it follows that:
In the even case the minimum distance is computed by $c_f$, for $f=X_{1,1}$, and also by $f=X_{1,1}+\delta X_{2,2}$ for any $\delta$ in $\Fq$. In the odd case it (follows from  Theorem \ref{thm:MinDist} that it) is computed by $f=X_{1,1}-\delta X_{2,2},$ for any $\delta$ a square in $\Fq$.

\end{corollary}

\section{Acknowledgment} 
Peter Beelen was supported by a research grant (VIL``52303”) from Villum Fonden.”

Trygve Johnsen was partially supported by the project Pure Mathematics in
Norway, funded by the Trond Mohn Foundation, and also by the UiT Aurora
project MASCOT. 

Prasant Singh was supported by SERB SRG grant SRG/2022/001643 from DST, Govt of India.

\end{document}